\documentclass[10pt, xcolor]{article}
\usepackage{amsfonts}
\usepackage{amssymb,amsthm, amsmath,graphicx,bm,amsthm}
\usepackage{mathrsfs}
\usepackage{color,enumerate,lineno}
\usepackage{url,hyperref}%,showlabels}
\usepackage{amsfonts}
\usepackage{amssymb,amsthm, amsmath,graphicx,bm,amsthm,mathtools}
\usepackage{color,enumerate}
\usepackage{url,lineno}
\usepackage{lscape}
\usepackage{multirow}
\usepackage{float}
\usepackage{fancyhdr}
\usepackage{emptypage}

\usepackage[latin1]{inputenc}

\usepackage[all]{xy}

\newtheorem{theorem}{Theorem}[section]

\newtheorem{proposition}[theorem]{Proposition}
\newtheorem{corollary}[theorem]{Corollary}
\newtheorem{lemma}[theorem]{Lemma}
\newtheorem{algorithm}[theorem]{Algorithm}
\theoremstyle{definition}

\newtheorem{note}[theorem]{Note}
\newtheorem{example}[theorem]{Example}

\def\a{{\mathrm a}}

\def\e{{\mathrm e}}
\def\g{{\mathrm g}}

\def\m{{\mathrm m}}
\def\n{{\mathrm n}}

\def\r{{\mathrm r}}

\def\gcd{{\mathrm{gcd}}}

\def\A{{\mathrm{A}}}

\def\F{{\mathrm F}}
\def\G{{\mathrm G}}

\def\Max{\mathrm{Max}}
\def\SG{\mathrm{SG}}

\def\PF{{\mathrm{PF}}}

\def\msg{{\mathrm{ msg }}}

\def\Ap{{\mathrm{ Ap}}}
\def\max{{\mathrm{ max}}}
\def\MED{{\mathrm{MED}}}
\def\Cad{{\mathrm {Cad}}}
\def\Rat{{\mathrm{Rat}}}

\def\MR{{\mathrm{MR}}}

\def\sA{\mathscr{A}}
\def\sR{\mathscr{R}}

\def\msg{{\mathrm{ msg }}}

\def\Maximals{\mathrm{Maximals}_{\leq_S}}
\def\max{\mathrm{max}}

\def\min{\mathrm{min}}
\def\max{\mathrm{max}}

\def\N{\mathbb{N}}

\def\Z{\mathbb{Z}}
\def\Q{\mathbb{Q}}

\def\rank{\mathrm{rank}\, }
\def\Ap{\mathrm{Ap}}
\def\int{\mathrm{int}}

\title{The ratio-covariety of numerical semigroups with fixed multiplicity and  Frobenius number}

\author{
	M. A. Moreno-Fr\'{\i}as \footnote{
		Dpto. de Matem\'aticas, Facultad de Ciencias,
		Universidad de C\'adiz, E-11510, Puerto Real  (C\'{a}diz, Spain).
		Partially supported by  Junta de Andaluc\'{\i}a group FQM-298, 
		Proyecto de Excelencia de la Junta de Andalucía ProyExcel\_00868, Proyecto de investigación del Plan Propio--UCA 2022-2023 (PR2022-011) and Proyecto de investigación del Plan Propio--UCA 2022-2023 (PR2022-004).
		E-mail: mariangeles.moreno@uca.es.}
	\and
	J. C. Rosales \footnote{
		Dpto. de \'Algebra, Facultad de Ciencias, Universidad de Granada,
		E-18071, Granada. (Spain).
		Partially supported by  Junta de Andaluc\'{\i}a group FQM-343,
		Proyecto de Excelencia de la Junta de Andalucía ProyExcel\_00868 and Proyecto de investigación del Plan Propio--UCA 2022-2023 (PR2022-011).
		E-mail: jrosales@ugr.es.}
}

\date{}

\begin{document}
 
\maketitle

\begin{abstract}
	
		In this work we will introduce the concept of ratio-covariety, as a nonempty family $\sR$ of numerical semigroups verifying  certain properties.   This  concept will allow us to:
		\begin{enumerate}
			\item Describe an algorithmic process to compute $\sR.$
			\item Prove the existence of the smallest element of $\sR$ that contains a set of positive integers.
			\item Talk about the smallest ratio-covariety that contains a finite set of numerical semigroups.
		\end{enumerate}
	In addition, in this paper we will apply the previous results to the study of the ratio-covariety $\sR(F,m)=\{S\mid S \mbox{ is a numerical semigroup with Fro-}\\
		 \mbox {benius number }F \mbox{ and multiplicity }m\}.$


   \smallskip
    {\small \emph{MSC-class:} 20M14 (Primary),  11D07, 13H10 (Secondary).}
\end{abstract}

\section{Introduction}
Let $\N$  be the set of nonnegative integers. A {\it numerical semigroup} $S$ is a submonoid of $(\N,+)$ such that $\N \backslash S$ is finite. The set $\N \backslash S$ is known as the set of {\it gaps} of $S$ and its cardinality is called the {\it genus} of $S$, denote by $\g(S).$ The largest integer not belonging to $S$ is the {\it Frobenius number} of $S$ and it will be denote by $\F(S).$ For instance, $\F(\N)=-1.$  
Let $\{n_1<\dots <\n_p\}\subseteq \N$ such that $\gcd(n_1,\dots, n_p)=1.$ Then 
$
\langle n_1,\dots, n_p \rangle =\left \{\sum_{i=1}^p \lambda_in_i \mid \{\lambda_1,\dots, \lambda_p\}\subseteq \N \right\}
$
is a numerical semigroup and every numerical semigroup has this form (see \cite[Lemma 2.1]{libro}). The set $\{n_1<\dots <n_p\}$ is called {\it system of generators} of $S$, and we write $S=\langle  n_1,\dots, n_p \rangle.$ 
We say that a system of generators of a numerical semigroup is a {\it minimal system of generators} if none of its proper subsets generates the numerical semigroup. Every numerical semigroup has a unique minimal system of generators, which in addition is finite (see  \cite[Corollary 2.8]{libro}). The minimal system of generators of a numerical semigroup $S$ is denoted by $\msg(S).$ Its cardinality is called the {\it embedding dimension} and will be denoted by $\e(S).$ Another invariant which we will use in this work is the {\it multiplicity} of $S$, denoted by $\m(S).$ It is defined as the minimum of $S\backslash \{0\}.$ 

The called Frobenius problem (see \cite{alfonsin}) for numerical semigroups is a classical mathematical problem. It lies in finding formulas to obtain the Frobenius number and the genus of a numerical semigroup from its minimal  system of generators. This problem was solved in \cite{sylvester} for numerical semigroups with embedding dimension two.
Since then, many researchers have tried to solve this problem for numerical semigroups with  embedding dimension greater than or equal to three. However,  the problem is still open. Furthemore, in this case the problem becomes NP-hard (see \cite{alfonsin2}).

Looking to find a solution to the Frobenius problem, we have  found common properties that verify certain families of numerical semigroups (see for instance \cite{covariedades} and \cite{coarf}).  

The motivation for  this work has been  to try to extend  the study of covarieties, which was initiated in \cite{covariedades}, to families of numerical semigroup with a fixed Frobenius number and multiplicity. That is, if $m$ and $F$ are positive integers such that $m<F$ and $m$ does not divide $F,$ then we denote  by  $$\sR(F,m)=\{S\mid S \mbox{ is a numerical semigroup, } \F(S)=F \mbox{ and }\m(S)=m \}.$$ The main aim of this paper is to study the set  $\sR(F,m).$ We can see the set $\sR(F,m)$ as a family of numerical semigroups which verifies some important properties. The generalization of these properties, leads us to the notion of {\it ratio-covariety}. 

For integers $a$ and $b,$ we say that $a$ {\it divides} $b$ if there exists an integer $c$ such that $b=ca,$ and we denote this by $a\mid b.$ Otherwise, $a$ {\it does not divide} $b$, and we denote this by $a\nmid b.$
Let $S$ be a numerical semigroup such that $S\neq \N,$ the {\it ratio }of $S$ is defined  $\r(S):=\min(\msg(S)\backslash \{\m(S)\}).$ Note that $\r(S)=\min\{s\in S\mid \m(S)\nmid s \}.$

A {\it ratio-covariety } is  a nonempty family $\sR$ of numerical semigroups   fulfilling  the following conditions:

\begin{enumerate}
	\item[1)]  There is the minimum of $\sR$, denoted by  $\min (\sR).$ 
	\item[2)] If $\{S, T\} \subseteq \sR$, then $S \cap  T \in \sR$.
	\item[3)]  If $S \in \sR$  and $S \neq  \min(\sR)$, then $S \backslash \{\r(S)\} \in \sR$.
\end{enumerate}

The organisation of the paper is the following.  Section 2 is devoted to prove that every ratio-covariety es finite and its elements can be ordered in a rooted tree. Moreover, we will describe the child  of an arbitrary vertex of the tree. As a consequence, in Section 3, we will present an algorithm which computes all the elements of $\sR(F,m).$

In Section 4, we focus on the characterization of the maximal elements of $\sR(F,m).$   This fact will allow to compute the set $\{\g(S)\mid S\in \sR(F,m)\}.$ We will apply all these result to present an algorithm which allow us to calculate all the elements of $\sR(F,m)$ with a fixed genus. 

Let $\sR$ be a ratio-covariety. A set $X$ will be call $\sR$-{\it set} if $X\cap \min(\sR)=\emptyset$ and there is  $S\in \sR$ such that $X\subseteq S.$ In Section 5, we will show  the  smallest element of $\sR$ containing $X.$ This element will be denote by $\sR[X]$ and we will say that it is the element of $\sR$ generated by $X.$

If $S=\sR[X],$ then we will say that $X$ is -{\it system of generators} of $S$. Moreover, if  $S\neq\ \sR[Y]$ for all $Y\subsetneq X,$ then $X$ will be called a {\it minimal} $\sR$-{\it system of generators} of $S.$ In Section 5, we will show an example of ratio-covariety   $\sR$ in which the minimal  $\sR$-system of generators is not unique. Also, we see in this section, that every element of $\sR(F,m)$ admits a unique minimal $\sR(F,m)$-system of generators.

Let  $\sR$ be a ratio-covariety and $S\in \sR,$ then the $\sR$-$\rank$ of $S$ is  $\sR\rank(S)=\min\{\sharp X\mid X \mbox{ is }\sR\mbox{-set and } \sR[X]=S\}$ (where $\sharp Y$, denote the cardinality of a set $Y$).  In Section 6, we will see that the maximum of $\{\sR(F,m)\rank(S)\mid S\in \sR(F,m)\}$ is  less than or equal to $m-2.$ This fact allows us to give the following definition. A numerical semigroup $S$ has {\it maximum rank} (hereinafter, $\MR$-{\it semigroup}) if $\F(S)>2\m(S)$ and $\sR(\F(S),\m(S))\rank(S)=\m(S)-2.$ In Section 6, we show a characterization of these kind of semigroups.

In Section 6, we will also see that if $S_1,\cdots, S_k$ are numerical semigroups with multiplicity $m$ and $F=\max\{\F(S_1),\cdots, \F(S_k)\},$ then there is the smallest  ratio-covariety containing the set $\{S_1,\cdots, S_k\}$ and  
being $\Delta(F,m)=\langle m \rangle \cup \{F+1,\rightarrow\}$ its minimum (the symbol $\rightarrow$ means that every integer greater than $F+1$ belongs to the set). This ratio-covariety is denoted by $\langle S_1,\cdots, S_k \rangle$ and we present an algorithm which computes all its elements.

Throughout this paper, some examples are shown to illustrate the results proven. For the development of these examples can be used the GAP(see \cite{GAP}) package \texttt{numericalsgps} (\cite{numericalsgps}).

\section{The tree associated with a ratio-covariety}

If $T$ is a numerical semigroup, then $\N\backslash T$ is a finite set and so we have the following result.
\begin{lemma}\label{lemma1}
Let $T$ be a numerical semigroup. Then  $$\{S\mid S \mbox{ is a numerical semigroup and }T\subseteq S\}$$ 
is a finite set. 
\end{lemma}

By applying Lemma \ref{lemma1} and that every ratio-covariety has a minimum, we obtain the following result. 
\begin{proposition}\label{proposition2}
	Every ratio-covariety has a finite cardinality.	
\end{proposition}

A {\it graph} $G$ is a pair $(V,E)$ where $V$ is a nonempty set and
$E$ is a subset of $\{(u,v)\in V\times V \mid u\neq v\}$. The
elements of $V$ and $E$ are called {\it vertices} and {\it edges},
respectively. A {\it path, of
	length $n$,} connecting the vertices $x$ and $y$ of $G$ is a
sequence of different edges of the form $(v_0,v_1),
(v_1,v_2),\ldots,(v_{n-1},v_n)$ such that $v_0=x$ and $v_n=y$.

A graph $G$ is {\it a tree} if there exists a vertex $r$ (known as
{\it the root} of $G$) such that for any other vertex $x$ of $G,$
there exists a unique path connecting $x$ and $r$. If  $(u,v)$ is an
edge of the tree $G$, we say that $u$ is a {\it child} of $v$.

Let $\sR$ be a ratio-covariety and $S\in \sR.$ Recurrently define the following sequence of elements of $\sR:$
\begin{itemize}
	\item $S_0 = S$,
	\item $S_{n+1} =\left\{\begin{array}{lcl}
		S_n \backslash \{\r(S_n)\} & &\mbox{if }\,\,  S_n\neq \min(\sR),\\
		\min(\sR) &  &\mbox{otherwise.}
	\end{array}
	\right.$
\end{itemize}
The following result has an immediate proof.

\begin{lemma}\label{lemma3}
	Let $\sR$ be a ratio-covariety, $S\in \sR$ and let $\{S_n\}_{n\in \N}$ be the sequence of elements of $\sR$ defined above. Then there exists $k\in \N$ such that $\min(\sR)=S_k\subsetneq S_{k-1}\subsetneq \cdots \subsetneq S_0=S.$ Moreover,  $\sharp(S_i\backslash S_{i+1})=1$ for all $i\in \{0,1,\cdots,k-1\}.$
\end{lemma}

If $\sR$ is a ratio-covariety, then we define the  graph $\G(\sR)$ as follows: $\sR$ is the set of vertices and 
$(S,T)\in \sR \times \sR$ is an edge of $\G(\sR)$ if and only if $T=S\backslash \{\r(S)\}.$

As a consequence from Lemma \ref{lemma3}, we have the following  result.

\begin{proposition}\label{proposition4} Let $\sR$ be a ratio-covariety.  Then
	$\G(\sR)$  is a tree with root $\min(\sR).$	
\end{proposition}

A tree can be  built recurrently starting from the root  and connecting, 
through an edge, the vertices already built with  their  children. Hence, it is very interesting to characterize the children of an arbitrary vertex in the tree.\\

Following the notation introduced in \cite{JPAA}, an integer $z$ is a {\it pseudo-Frobenius number} of a numerical semigroup $S$ if $z\notin S$ and $z+s\in S$ for all $s\in S\backslash \{0\}.$  We denote by $\PF(S)$
the set formed by the  pseudo-Frobenius numbers of $S.$ The cardinality of $\PF(S)$ is an important invariant of $S$ (see \cite{froberg} and \cite{barucci}) called the {\it type} of $S.$  
	
	Given $S$  a numerical semigroup, we denote by  $\SG(S)=\{x\in \PF(S)\mid 2x \in S\}.$ Its elements will be called {\it special gaps} of $S.$ The following result is \cite[Proposition 4.33]{libro}.
\begin{lemma}\label{lemma6}
	Let $S$ be a numerical semigroup and $x\in \N\backslash S.$ Then $x\in \SG(S)$ if and only if $S \cup \{x\}$ is a numerical semigroup.
\end{lemma}

\begin{example}\label{example8}Let $S=\langle 3,5,7 \rangle=\{0,3,5, \rightarrow\},$ then it is clear that $\PF(S)=\{2,4\}$ and $\SG(S)=\{4\}.$
\end{example}

The above calculations can be performed using the following gap orders:
\begin{verbatim}
	gap> S := NumericalSemigroup(3,5,7);
	<Numerical semigroup with 3 generators>
	gap> PseudoFrobenius(S);
	[ 2, 4 ]
	gap> SpecialGaps(S);
	[ 4 ]	
\end{verbatim}

\begin{proposition}\label{proposition7}
	Let $\sR$ be a ratio-covariety and $S\in \sR.$ Then the set formed by the children of $S$ in the tree $\G(\sR)$ is 
	$$
	\{S\cup \{x\}\mid x\in \SG(S),\, \m(S)<x<\r(S) \mbox{ and }S \cup \{x\}\in \sR\}.$$
		
\end{proposition}
\begin{proof}
	If $T$ is a child of $S,$ then $T\in \sR $ and $T\backslash  \{\r(T)\}=S.$ Therefore, $T=S\cup \{\r(T)\}\in \sR,$ $\r(T)\in \SG(S)$ and $\m(S)<\r(T)<\r(S).$
	
	If $\m(S)<x<\r(S),$ then $\r(S\cup \{x\})=x.$ Hence, $S\cup \{x\}\in \sR$ and $\left(S\cup \{x\} \right)\backslash \{\r(S\cup \{x\})\}=S.$ Consequently, $S\cup \{x\}$ is a child of $S$ in the tree $\G(\sR).$
\end{proof}

%For instance, to compute $\Ap(S,8),$ being $S=\langle 8,9,11,13 \rangle,$  we use  the following sentences:
%\begin{verbatim}
%	gap> S := NumericalSemigroup(8,9,11,13);
%	<Numerical semigroup with 4 generators>
%	gap> AperyList(S,8);
%	[ 0, 9, 18, 11, 20, 13, 22, 31 ]
%\end{verbatim}
%

\section{An algorithm for  computing  $\sR(F,m)$}
Throughout this section, $m$ and $F$ denote positive integers such that $m<F$ and $m\nmid F$. Recall that $$\sR(F,m)=\{S\mid S \mbox{ is a numerical semigroup, } \F(S)=F \mbox{ and }\m(S)=m \}.$$

The following result has an easy proof.
\begin{lemma}\label{lemma9} 
With the above notation, we have that $$\Delta(F,m)=\langle m \rangle \cup \{F+1,\rightarrow\}$$ is the minimum of $\sR(F,m).$	
\end{lemma}
The next result is well known and easy to demonstrate. 

\begin{lemma}\label{lemma10}
	Let $S$ and $T$ be a numerical semigroups and $x\in S.$ Then the following conditions hold:
	\begin{enumerate}
		\item $S\cap T$ is a numerical semigroup and $\F(S\cap T)=\max\{\F(S),\F(T)\}.$
		\item $S\backslash \{x\}$ is a numerical semigroup if and only if $x\in \msg(S).$
	\end{enumerate}
\end{lemma} 
According to Lemmas \ref{lemma9} and \ref{lemma10},  we can easily deduce the following result.

\begin{proposition}\label{proposition11}
	$\sR(F,m)$ is a ratio-covariety.
\end{proposition} 
As a consequence of Propositions \ref{proposition7} and \ref{proposition11}, we can charaterize the children of the tree $\G(\sR(F,m)).$
\begin{proposition}\label{proposition12}
	If $S\in \sR(F,m),$ then the  children of $S$ in the tree $\G(\sR(F,m))$ is the set
	$
	\{S\cup \{x\}\mid x\in \SG(S),\, m<x<\r(S) \mbox{ and }x\neq F \}.
	$ 
	
\end{proposition}
Given a numerical semigroup $S$ and $s\in S\backslash \{0\}$,  the{\it Apéry set} of $n$ in $S$ (in honour of \cite{apery}) is the set $\Ap(S,n)=\{s\in S\mid s-n \notin S\}$. This set has  $n$ elements, one per every congruence class modulo $n$. That is, $\Ap(S,n)=\{0=w(0),w(1), \dots, w(n-1)\}$, where $w(i)$ is the least
element of $S$ congruent with $i$ modulo $n$, for all $i\in
\{0,\dots, n-1\}$ (see \cite[Lemma 2.4]{libro}).

Given the numerical semigroup $S=\langle 6,8,13,15\rangle$, the set $\Ap(S,6)$  can be computed using the following gap orders:
\begin{verbatim}
	gap> S := NumericalSemigroup(6,8,13,15);
	<Numerical semigroup with 4 generators>
	gap> AperyList(S,6);
	[ 0, 13, 8, 15, 16, 23 ]
\end{verbatim}

Let $S$ be a numerical semigroup. We define over $\Z$ the  following order  relation:  $a\leq_S b$ if $b-a \in S.$ The following result is Lemma 10 from \cite{JPAA}.
\begin{lemma}\label{lemma14} If $S$ is  a numerical semigroup and $n \in S\backslash
	\{0\},$ then
	$$
	\PF(S)=\{w-n\mid w \in \Maximals \Ap(S,n)\}.
	$$
\end{lemma}
The next lemma  has an immediate proof.

\begin{lemma}\label{lemma15}
	Let $S$ be a numerical semigroup, $n\in S\backslash \{0\}$ and $w \in \Ap(S,n).$ Then $w\in \Maximals(\Ap(S,n))$ if and only if $w+w'\notin
	\Ap(S,n)$ for all $w'\in \Ap(S,n)\backslash \{0\}. $	
\end{lemma}

It is  straightforward to prove the following result.

\begin{lemma}\label{lemma16} If $S$ is a numerical semigroup and $S\neq \N,$ then $$\SG(S)=\{x\in \PF(S)\mid 2x \notin \PF(S)\}.$$
\end{lemma}

\begin{note}\label{note17}
	As a consequence of Lemmas \ref{lemma14}, \ref{lemma15} and \ref{lemma16}, observe that if $S$ is a numerical semigroup and we know $\Ap(S,n)$ for some $n\in S\backslash \{0\},$ then we can compute easily the set $\SG(S).$	
\end{note}
We will illustrate the previous note with an example.
\begin{example}\label{example18}
	If $S=\{0,7,\rightarrow\},$ then $\Ap(S,7)=\{0,8,9,10,11,12,13\}.$ By applying Lemma \ref{lemma15}, we have that $\Maximals(\Ap(S,7))=\{8,9,10,11,12,13\}.$ Then Lemma \ref{lemma14}, asserts that $\PF(S)=\{1,2,3,4,5,6\}.$ Finally, by using Lemma \ref{lemma16}, we have that $\SG(S)=\{4,5,6\}.$	
\end{example}

The following result has an easy proof.
\begin{lemma}\label{lemma19}
	If $S$ is a numerical semigroup, $n\in S \backslash \{0\}$ and $x\in \SG(S),$ then $x+n\in \Ap(S,n).$ Moreover, $\Ap(S\cup \{x\},n)=\left( \Ap(S,n)\backslash \{x+m\}\right)\cup \{x\}.$
\end{lemma}

\begin{note}\label{note20} Observe that as a consequence of Lemma \ref{lemma19}, if we know $\Ap(S,n),$ then we can easily compute $\Ap(S\cup \{x\},n).$ In particuar, if $\sR$ is a ratio-covariety and $S\in \sR,$ then Lemma \ref{lemma19} allows us to compute the set $\Ap(T,n)$ from $\Ap(S,n),$  for every child $T$ of $S$ in the tree $\G(\sR)$ (see Proposition \ref{proposition7}).	
\end{note}We are going to illustrate the previous note with an example.
\begin{example}\label{example21}
	Consider the numerical semigroup $S=\{0,7,\rightarrow\}.$ We have that $\Ap(S,7)=\{0,8,9,10,11,12,13\}.$ By Example \ref{example18}, we know that $5\in \SG(S).$ If $T=S\cup \{5\},$ then  Lemma \ref{lemma19} asserts that 
	$$
	\Ap(T,7)=\left( \{0,8,9,10,11,12,13\}\backslash \{5+7\}\right)\cup \{5\}=\{0,5,8,9,10,11,13\}.
	$$
	
\end{example}

We already have all the necessary tools to present the algorithm that gives title to this section.

\begin{algorithm}\label{algorithm22}\mbox{}\par
\end{algorithm}
\noindent\textsc{Input}: Two positive integer $F$ and $m$ such that $m<F$ and $m\nmid F.$   \par
\noindent\textsc{Output}: $\sR(F,m).$

\begin{enumerate}
	\item[(1)] Compute $\Ap(\Delta(F,m),m).$ 
	\item[(2)] $\sR(F,m)=\{\Delta(F,m)\}$ and $B=\{\Delta(F,m)\}.$
	\item[(3)] For every $S\in B$ compute $\theta(S)=\{x\in \SG(S)\mid m<x<\r(S) \mbox{ and }x\neq F\}.$
	\item[(4)] If   $\displaystyle\bigcup_{S\in B}\theta(S)=\emptyset,$ then return $\sR(F,m).$ 	
	\item[(5)] $C=\displaystyle\bigcup_{S\in B}\{S\cup \{x\}\mid x\in \theta(S)\}.$
	\item[(6)]  $\sR(F,m)= \sR(F,m)\cup C$ and  $B=C.$  
	\item[(7)] For every $S\in B,$ compute $\Ap(S,m)$ and go to Step $(3).$ 	
\end{enumerate}
Next we illustrate this algorithm with an example.

\begin{example}\label{example23}
	We are going to compute $\sR(7,4)$ by using Algorithm \ref{algorithm22}.
	\begin{itemize}
		\item $\Delta(7,4)=\{0,4,8,\rightarrow\}$ and $\Ap(\Delta(7,4),4)=\{0,9,10,11\}.$
		\item $\sR(7,4)=\{\Delta(7,4)\}$ and $B=\{\Delta(7,4)\}.$
		\item $\theta(\Delta(7,4))=\{5,6\}.$
		\item $C=\left\{\Delta(7,4)\cup \{5\},\Delta(7,4)\cup \{6\} \right\}.$
		\item $\sR(7,4)=\{\Delta(7,4),\Delta(7,4)\cup \{5\},\Delta(7,4)\cup \{6\}\}$ and $B=\{\Delta(7,4)\cup \{5\},\Delta(7,4)\cup \{6\}\}.$
		\item $\Ap\left(\Delta(7,4)\cup \{5\},4 \right)=\{0,5,10,11\}$ and $\Ap\left(\Delta(7,4)\cup \{6\},4 \right)=\{0,6,9,11\}.$
		\item $\theta(\Delta(7,4)\cup \{5\} )=\emptyset$ and $\theta(\Delta(7,4)\cup \{6\})=\{5\}.$
		\item $C=\left\{\Delta(7,4)\cup \{5,6\} \right\}.$
		\item $\sR(7,4)=\{\Delta(7,4),\Delta(7,4)\cup \{5\},\Delta(7,4)\cup \{6\}, \Delta(7,4)\cup \{5,6\}\}$ and $B=\{\Delta(7,4)\cup \{5,6\}\}.$
		\item $\Ap\left(\Delta(7,4)\cup \{5,6\},4 \right)=\{0,5,6,11\}.$ 
		\item $\theta(\Delta(7,4)\cup \{5,6\})=\emptyset.$
		\item The Algorithm returns $$\sR(7,4)=\{\Delta(7,4),\Delta(7,4)\cup \{5\},\Delta(7,4)\cup \{6\}, \Delta(7,4)\cup \{5,6\}\}.$$
	\end{itemize}
	
	\end{example}

\section{The elements of $\sR(F,m)$ with a fixed genus}
Let $F,$ $m$ and $g$ be  positive integers. Denote by $\sR(F,m,g)=\{S\in \sR(F,m)\mid \g(S)=g\}.$ The following result can be deduced from \cite[Lemma 2.14]{libro}.
\begin{proposition}\label{lemma24}
	If $S$ is a numerical semigroup, then $\frac{\F(S)+1}{2}\leq \g(S)\leq \F(S).$
\end{proposition}

A numerical semigroup is {\it irreducible} if it cannot be expressed
as the intersection of two numerical semigroups properly containing
it. This notion was introduced in \cite{irreducibles}, where also the next result is proven.
\begin{lemma}
	Let $S$ be a numerical semigroup. Then $S$ is irreducible if and only if $S$ is a maximal element in the set $\{T\mid T \mbox{ is a numerical semigroup and } \F(T)=\F(S)\}.$
\end{lemma}

The irreducible numerical semigroups are very interesting because from \cite{barucci} and \cite{froberg} can be deduced that a numerical semigroup is irreducible if and only if it is symmetric or pseudo-symmetric.
This kind of numerical semigroup has been widely studied in the literature because one dimensional analytically irreducible  local ring is Gorenstein (respectively Kunz)
if and only if its value semigroup is symmetric (respectively pseudo-symmetric), as it can be seen in \cite{kunz} and \cite{barucci}.

If $q\in \Q,$ then $\lceil q \rceil=\min\{z\in \Z\mid q\leq z\}.$  From \cite[Corollary 4.5]{libro} we have the following characterization.

\begin{lemma}\label{lemma26}
Let $S$ be a numerical semigroup. Then the following conditions hold.
\begin{enumerate}
	\item $S$ is symmetric if and only if $\g(S)=\displaystyle \frac{\F(S)+1}{2}.$
	\item $S$ is pseudo-symmetric if and only if  $\g(S)=\displaystyle \frac{\F(S)+2}{2}.$
	\item $S$ is irreducible if and only if  $\g(S)=\displaystyle \left \lceil \frac{\F(S)+1}{2}\right\rceil .$
\end{enumerate}
\end{lemma}

Denote by $\Max(\sR(F,m))$ the set formed by the maximal elements of $\sR(F,m).$ The following result is deduced from \cite[Proposition 6]{IJAC}.

\begin{proposition}\label{proposition27}With the above notation, we have the following.
	\begin{enumerate}
		\item If $F=m-1,$ then $\Max(\sR(F,m))=\sR(F,m)=\{\{0,m,\rightarrow\}\}.$
		\item If $m<F<2m,$ then $\Max(\sR(F,m))=\{\{0,m,\rightarrow\}\backslash \{F\}\}.$
		\item If $F>2m,$ then $\Max(\sR(F,m))=\{S\in \sR(F,m)\mid S \mbox{ is irreducible}\}.$
		\end{enumerate}
	
\end{proposition}
As a consequence from Lemma \ref{lemma26} and Proposition \ref{proposition27}, we have the following result.
\begin{corollary}\label{corollary28}
	Let $S\in \sR(F,m).$ Then $S\in \Max(\sR(F,m))$ if and only if one of the following conditions are verified.
	\begin{enumerate}
		\item $F=m-1$ and $\g(S)=m-1.$
		\item $m<F<2m$ and $\g(S)=m.$
		\item $F>2m$ and $\g(S)=\left \lceil \displaystyle \frac{F+1}{2}  \right \rceil.$
	\end{enumerate}
	
\end{corollary}

If $S$ is a numerical semigroup such that $S\neq \N,$ then the {\it ratio-sequence associated} to $S$ is recurrentely defined as: $S_0=S$ and $S_{n+1}=S_n\backslash \{\r(S_n)\}$ for all $n\in \N.$

If $S$ is a numerical semigroup, then we denote by $\A(S)=\{x\in S\mid x<\F(S) \mbox{ and }\m(S)\nmid x\}.$ The cardinality of $\A(S)$ will be denoted by $\a(S).$

Let $S$ be a numerical semigroup and  let $\{S_n\}_{n\in \N}$ be the  ratio-sequence associated to $S$, then the set  $\Rat$-$\Cad(S)=\{S_0,S_1,\dots, S_{\a(S)}\}$ is called the {\it  ratio-chain associated} to $S.$ It is clear that 
 $S_{\a(S)}=\Delta(\F(S), \m(S)).$

 If $q\in \Q,$ then $\lfloor q \rfloor=\max\{z\in \Z\mid z\leq q\}.$
 \begin{lemma}\label{lemma29}
 	With the above notation, it is verified that $\g(\Delta(F,m))=F-\left \lfloor \displaystyle \frac{F}{m}\right \rfloor.$
 \end{lemma}
\begin{proof}
	It is enough to observe that $$\Delta(F,m)=\{0,m,2m,\cdots, \left \lfloor\displaystyle \frac{F}{m}\right \rfloor m, F+1,\rightarrow \}.$$
\end{proof}

The following result has an easy proof.

\begin{lemma}\label{lemma30}
	If $m<F<2m,$ then $$\sR(F,m)=\{A\cup \{0,m,F+1,\rightarrow\}\mid A\subseteq \{m+1,\cdots, F-1\}\}.$$
	
\end{lemma}

By applying the previous results, we can easily deduced the following proposition. 
\begin{proposition}\label{proposition31}
	With the above notation, the following holds:
	\begin{enumerate}
		\item If $F=m-1,$ then $\{\g(S)\mid S\in \sR(F,m)\}=\{m-1\}.$
		\item If $m<F<2m,$ then $\{\g(S)\mid S\in \sR(F,m)\}=\{m, m+1,\cdots,F-1\}.$
		\item If $F>2m,$ then $\{\g(S)\mid S\in \sR(F,m)\}=\left\{ \left \lceil \displaystyle  \frac{F+1}{2} \right \rceil , \cdots, F-\left\lfloor\displaystyle \frac{F}{m} \right \rfloor \right\}.$
	\end{enumerate}
	
\end{proposition}

We have now all the ingredients needed to give an algorithmic procedure to compute $\sR(F,m,g),$ being $F>2m$, $m\nmid F$ and $\displaystyle \left \lceil\frac{F+1}{2} \right \rceil \leq g \leq F-\left\lfloor\frac{F}{m}\right \rfloor.$

\begin{algorithm}\label{algorithm32}	
\end{algorithm}

\noindent\textsc{Input}: Positive integer $F$, $m$ and $g$ such that $F>2m$, $m\nmid F$ and $\displaystyle \left \lceil \frac{F+1}{2}\right \rceil\leq g \leq F-\left\lfloor\frac{F}{m}\right \rfloor.$  \par
\noindent\textsc{Output}: $\sR(F,m,g).$

\begin{enumerate}
	\item[(1)] $H=\{\Delta(F,m)\},$ $i=F-\left\lfloor\frac{F}{m}\right \rfloor.$ 
	\item[(2)] If $i=g$, return $H.$
	\item[(3)] For every $S\in H$ compute $\theta(S)=\{x\in \SG(S)\mid m<x<\r(S) \mbox{ and }x\neq F\}.$	
	\item[(5)] $H:=\displaystyle\bigcup_{S\in H}\{S\cup \{x\}\mid x\in \theta(S)\},$ $i:=i-1$ and go to Step(2).
\end{enumerate}
We are going to see how the previous algorithm works in the following example. 
\begin{example}\label{example33}
By using Algorithm \ref{algorithm32}, we will compute the set  $\sR(12,5,8).$
\begin{itemize}
	\item $H=\{\Delta(12,5)\},$ $i=10.$
	\item $\theta(\Delta(12,5))=\{8,9,11\}.$
	\item $H=\{\Delta(12,5)\cup \{8\}, \Delta(12,5)\cup \{9\},\Delta(12,5)\cup \{11\}\},$ $i=9.$
	\item $\theta(\Delta(12,5)\cup \{8\})=\emptyset,$ $\theta(\Delta(12,5)\cup \{9\})=\{8\}$ and $\theta(\Delta(12,5)\cup \{11\})=\{8,9\}.$
	\item $H=\{\Delta(12,5)\cup \{8,9\}, \Delta(12,5)\cup \{8,11\},\Delta(12,5)\cup \{9,11\}\},$ $i=8.$
	\item The algorithm returns $$\sR(12,5,8)=\{\Delta(12,5)\cup \{8,9\}, \Delta(12,5)\cup \{8,11\},\Delta(12,5)\cup \{9,11\}\}.$$	
\end{itemize}	
\end{example}

We finish this section, noting that in \cite{portuguesa} an equivalence relation $\sim$ is defined over  $\sR(F,m)$ such that the set  $\displaystyle \frac{\sR(F,m)}{\sim}=\{[S]\mid S\in \Max(\sR(F,m))\}$ is the quotient set of $\sR(F,m)$ by $\sim.$
 As a consequence, in \cite{portuguesa} an algorithm is obtained to compute $\sR(F,m)$ based on two algorithms:
 \begin{enumerate}
 	\item An algorithm which computes $\Max(\sR(F,m)).$
 	\item Another algorithm which computes $[S]$ when  $S$ belongs to $\Max(\sR(F,m)).$
 \end{enumerate}

\section{$\sR$-system of generators}
Throughout this section $\sR$ will denote a ratio-covariety. We will say that a set $X$ is an $\sR$-{\it set}, if it verifies the following conditions:
\begin{enumerate}
	\item[1)]$X\cap \min(\sR)=\emptyset.$
	\item[2)]There is $S\in \sR$ such that $X\subseteq S.$	
\end{enumerate} 

If $X$ is an $\sR$-set, then we will denote by $\sR[X]$  the intersection of all elements of $\sR$ containing $X.$ By Proposition \ref{proposition2}, we know that $\sR$ is a finite set and so, the intersection of elements of $\sR$ is an element of $\sR.$ Therefore, we can ennounce the following proposition.
\begin{proposition}\label{proposition34}
	Let $X$ be an $\sR$-set. Then $\sR[X]$ is the smallest element of $\sR$ containing $X.$
\end{proposition}

If $X$ is an $\sR$-set and $S=\sR[X],$ we will say that $X$ is an $\sR$-{\it system of generators} of $S.$ Moreover, if $S\neq \sR[Y]$ for all $Y\subsetneq X,$ then $X$ will be called  a {\it minimal} $\sR$-{\it system of generators} of $S.$

In general, the minimal  $\sR$-system of generators is not unique as we can see in the following example.  
\begin{example}\label{example35}
	Let $\sR=\{S\mid S \mbox{ is a numerical semigroup, } \m(S)=8,\, \r(S)\ge 10 \mbox{ and }\F(S)\leq 15 \}\cup \{S_1=\{0,8,9,10,11,15,\rightarrow\}, S_2=\{0,8,9,13,15,\rightarrow\}, S_3=\{0,8,9,15,\rightarrow\}\}.$
	
	It is easy to see that $\sR$ is a ratio-covariety, $\sR([\{9,10\}])=\sR([\{9,11\}])=S_1,$ $\sR([\{9\}])=S_3,$ $\sR([\{10\}])=\{0,8,10,15,\rightarrow\}$ and $\sR([\{11\}])=\{0,8,11,15,\rightarrow\}.$
	Therefore, $\{9,10\}$ and  $\{9,11\}$ are minimal $\sR$-system of generators of $S_1.$ 	
\end{example}

Our next aim in this section will be to prove that every element of $\sR(F,m)$ admits a unique minimal $\sR(F,m)$-system of generators. 

\begin{lemma}\label{lemma36} If $S\in \sR,$ then $X=\{x\in \msg(S)\mid x\notin \min(\sR)\}$ is an $\sR$-set and $\sR[X]=S.$
\end{lemma}
\begin{proof}
	It is clear that $X$ is an $\sR$-set. As $S\in \sR$ and $X\subseteq S,$ then $\sR[X]\subseteq S.$ We are going to see the reverse inclusion. Let $T\in \sR$ such that $X\subseteq T.$ Then $X\cup \min(\sR)\subseteq T$ and so $\msg(S)\subseteq T.$ Hence, $S\subseteq T$ and consequently, $S\subseteq \sR[X].$
\end{proof}
\begin{proposition}\label{proposition37}
If $S\in \sR(F,m),$ then $X=\{x\in \msg(S)\mid x\notin \Delta(F,m)\}$ is the unique minimal $\sR$-system of generators of $S.$
\end{proposition}
\begin{proof}
	By Lemma \ref{lemma36}, we know that $X$ is a $\sR(F,m)$-set and $\sR(F,m)[X]=S.$ To conclude the proof, we will see that if $Y$ is a $\sR(F,m)$-set and  $\sR(F,m)[Y]=S,$ then $X\subseteq Y.$ In fact, if $X\nsubseteq Y,$ then there exists $x\in X\backslash Y.$ By applying Lemma \ref{lemma10}, we deduce that $S\backslash \{x\}\in \sR(F,m)$ and $Y\subseteq S\backslash \{x\}.$ Therefore, $S=\sR(F,m)[Y]\subseteq S\backslash \{x\}$, which is absurd.
\end{proof}
If $\sR$ is a ratio-covariety and $S\in \sR,$ then we define the $\sR$-{\it rank} of $S$ as 
$$
\sR\rank(S)=\min\{\sharp X\mid X \mbox{ is a }\sR\mbox{-set and }\sR[X]=S\}.
$$

As an immediate consequence of Lemma \ref{lemma36}, in the following result  we show the relation between the $\sR$-rank and the embedding dimension of a numerical semigroup. 

\begin{proposition}\label{proposition38}
	If $\sR$ is a ratio-covariety and $S\in \sR,$ then $\sR\rank(S)\leq \e(S).$
\end{proposition}
The proof of the following result is inmmediate.
\begin{lemma}\label{lemma39}
Let $\sR$ be a ratio-covariety and $S\in \sR.$ Then $\sR\rank(S)=0$ if and only if $S=\min(\sR).$
\end{lemma}
\begin{lemma}\label{lemma40}
If $\sR$ is a ratio-covarity, $S\in \sR,$ $S\neq \min(\sR)$ and $X$ is an $\sR$-set such that $S=\sR[X],$ then $\r(S)\in X.$
\end{lemma}
\begin{proof}
	If $\r(S)\notin X,$ then $X\subseteq S \backslash \{r(S)\}.$ As $S\backslash \{\r(S)\}\in \sR,$ then $S=\sR[X]\subseteq S\backslash \{\r(S)\}$ which is absurd.
\end{proof}

As a consequence of Lemmas \ref{lemma39} y \ref{lemma40}, in the following result we characterize the numerical semigroups $S\in \sR$ such that  $\sR\rank(S)=1.$ 

\begin{proposition}\label{proposition41}Let $\sR$ be a ratio-covariety and $S\in \sR.$ Then  $\sR\rank(S)=1$ if and only if $S=\sR[\{\r(S)\}].$
\end{proposition}
As a consequence of Propositions \ref{proposition37} and \ref{proposition41}, we have the following result.
\begin{corollary}\label{corollary42}
	Let $m<r<F$ be positive integers such that $m\nmid r$ and $F\notin \langle m,r \rangle.$ Then $\langle m,r \rangle \cup \{F+1,\rightarrow\}$ is an element of $\sR(F,m)$ with $\sR(F,m)$-$\rank$ equal to one. Moreover, every element of $\sR(F,m)$ with $\sR(F,m)$-$\rank$ equal to one, has this form.	
\end{corollary}
\section{The elements of $\sR(F,m)$ with  maximun $\sR(F,m)$-$\rank$ }
Our first goal in this section will be to show that the maximum of the set $\{\sR(F,m)\rank(S)\mid S\in \sR(F,m)\}$ is less than or equal to $m-2.$ For this reason, we need to introduce some concepts and results.

The following result is well known and it appears in \cite[Proposition 2.10]{libro}.

\begin{lemma}\label{lemma43}
If $S$ is a numerical semigroup, then $\e(S)\leq \m(S).$	
\end{lemma}

A numerical semigroup S is said to have {\it maximal embedding
	dimension} (from now on $\MED$-{\it semigroup}) if $\e(S) = \m(S).$ 

In the literature one can find a long list of works dealing with the study
of one dimensional analytically irreducible local domains via their value
semigroups. One of the properties
studied for this kind of rings using this approach is that of being of maximal
embedding dimension (see \cite{abhyankar},\cite{barucci}, \cite{brown-herzog} and \cite{sally}).
 The characterization of ring with maximal embedding dimension via their value semigroup, gave rise to the notion of $\MED$-semigroup.

 From \cite[Corollary 3.2]{libro} we can deduce the following result.

\begin{lemma}\label{lemma44}If $S$ is a $\MED$-semigroup and $S\neq \N$, then $\F(S)=\max(\msg(S))-\m(S).$
\end{lemma}
\begin{proposition}\label{proposition45}If $S\in \sR(F,m),$ then $\sR(F,m)\rank(S)\leq m-2.$
\end{proposition}
\begin{proof}
	By Proposition \ref{proposition37}, we know that $\sR(F,m)\rank(S)$ is the cardinality of the set $X=\{x\in \msg(S)\mid x\notin \Delta(F,m)\}.$ As $m\in \msg(S)$ and $m\in \Delta(F,m),$ then $\sR(F,m)\rank(S)\leq \e(S)-1.$ Now, by applying Lemma \ref{lemma43}, we have that $\sR(F,m)\rank(S)\leq \m(S)-1.$ To finish the proof, we will see that the case $\sR(F,m)\rank(S)= \m(S)-1$ is impossible. Indeed, if $\sR(F,m)\rank(S)= \m(S)-1,$ then we deduce that $\e(S)=m$ and so $S$ is a $\MED$-semigroup. But, if $S$ is a $\MED$-semigroup, then by applying Lemma \ref{lemma44}, we have that $\max(\msg(S))\in \Delta(F,m).$ Now, by Proposition \ref{proposition37}, we have that $\sR(F,m)\rank(S)\leq \e(S)-2\leq m-2.$
\end{proof}
The following result is easily deduced from  Proposition \ref{proposition37}.
\begin{proposition}\label{proposition46}
If $\{m=a_1<a_2<\cdots <a_k<a_{k+1}=F\}\subseteq \N$ and $a_{i+1}\notin \langle a_1,\cdots,a_i \rangle$ for all $i\in \{1,\cdots,k\},$ then $\langle a_1,\cdots,a_k\rangle \cup \{F+1,\rightarrow\}$ is an element of 
$\sR(F,m)$ with $\sR(F,m)$-$\rank$ equal to $k-1.$ Moreover, every element of $\sR(F,m)$ with $\sR(F,m)$-$\rank$ equal to $k-1$ has this form.
\end{proposition}

\begin{corollary}\label{corollary47}
	If $S\in \sR(F,m)$ and $\sR(F,m)\rank(S)=m-2,$ then $F\ge 2m-1.$ Moreover, if $F=2m-1,$ then $S=\{0,m,m+1,\cdots, 2m-2,2m, \rightarrow\}.$
\end{corollary}
\begin{proof}
	By Proposition \ref{proposition46}, it is straightforward to see that $F\ge 2m-1.$ By applying now Lemma \ref{lemma30} and Proposition \ref{proposition46}, we have that  $S=\{0,m,m+1,\cdots, 2m-2,2m, \rightarrow\}.$ 
\end{proof}

\begin{proposition}\label{proposition48} If $F>2m,$ then the set $\{S\in \sR(F,m)\mid  \sR(F,m)\rank(S)=m-2\}\neq \emptyset.$	
\end{proposition}
\begin{proof}
	Let $S=\langle m, F-(m-1), F-(m-2),\cdots, F-1\rangle \cup \{F+1,\rightarrow\}.$ As $F>2m,$ then $F-(m-1)>m.$ As $m\nmid F,$ then there is a unique $i\in \{1,\cdots,m-1\}$ such that $m\mid F-i.$ Then 
	we easily deduce that $\{x\in \msg(S)\mid x\notin \Delta(F,m)\}=\{F-(m-1),\cdots,F-1\}\backslash \{F-i\}.$ By applying now  Proposition \ref{proposition37}, we have that $\sR(F,m)\rank(S)=m-2.$	
\end{proof}

The above propostion allow us to define a new kind of semigroups. 

A numerical semigroup $S$ is said to have {\it maximal rank} 
	 (hereinafter $\MR$-{\it semigroup} ) if $\F(S)>2\m(S)$ and  $\sR(\F(S),\m(S))\rank(S)=\m(S)-2.$
	 
	 The existence of $\MR$-semigroups is assured by  Proposition \ref{proposition46}. Now we present a characterization of these semigroups which  is also a direct consequence of Proposition \ref{proposition46}.

\begin{proposition}\label{proposition49} Let $S\in \sR(F,m)$ with $F>2m.$ Then $S$ is a $\MR$-semigroup if and only if one  of the following statements is true:
	\begin{itemize}
		\item $\e(S)=m-1$ and $\msg(S)\subseteq \{1,\cdots, F\}.$
		\item $\e(S)=m$ and  $\sharp\{x\in \msg(S)\mid x>F\}=1.$
	\end{itemize}
\end{proposition}

Next we illustrate this characterization with an example.

\begin{example}\label{example50}
	\begin{enumerate}
	\item Let $S=\langle 5,7,9,11\rangle.$ Then $S\in \sR(13,5),$ $\e(S)=4$ and $\msg(S)\subseteq\{1, \cdots, 13\}.$ Therefore, by applying Proposition \ref{proposition49}, we have that $S$ is a $\MR$-semigroup.
	\item If $S=\langle 5,12,13,14,21\rangle,$ then $S\in \sR(16,5),$ $\e(S)=5$ and $\{x\in \msg(S)\mid x>16\}=\{21\}.$ Consequently,  by applying Proposition \ref{proposition49}, we can assert  that $S$ is a $\MR$-semigroup.
	\end{enumerate}
\end{example}

\section{The ratio-covariety generated by a finite family of numerical semigroups}

In general, the intersection of ratio-covarieties is not a ratio-covariety. Actually, if $S$ and $T$ are numerical semigroups such that  $S\neq T,$ then $\sR_1=\{S\}$ and $\sR_2=\{T\}$ are ratio-covarieties. However, $\sR_1\cap \sR_2=\emptyset,$ which is not a ratio-covariety.

The following result has an immediate proof.

\begin{lemma}\label{lemma51}
	If $\{\sR_i\}_{i\in I}$ is a family of ratio-covarieties and $\min(\sR_i)=\Delta$ for every $i\in I,$ then $\displaystyle\cap_{i\in I}\sR_i$ is a ratio-covariety.	
\end{lemma}
Let $m$ and $F$ be positive integers such that $m<F$ and $m\nmid F,$ then we denote by $\sA(F,m)=\{S\mid S \mbox{ is a numerical semigroup, }\F(S)\leq F \mbox{ and }\m(S)=m\}.$

The proof of the following result is easy.

\begin{lemma}\label{lemma52} With the above notation, $\sA(F,m)$ is a ratio-covariety being $\Delta(F,m)$ its minimum.
\end{lemma}
The next lemma has an immediate proof. 
\begin{lemma}\label{lemma53}
	Let $S_1,\cdots, S_k$ be numerical semigroups with multiplicity $m$ and let $F=\max \{\F(S_1),\cdots, \F(S_k) \}.$ Then $\{S_1,\cdots, S_k\}\subseteq \sA(F,m).$
\end{lemma}

Let  $\{S_1,\cdots, S_k\}$ be   a finite set of numerical semigroups with multiplicity $m$ and $F=\max \{\F(S_1),\cdots, \F(S_k)\}.$ Then we denote by $\langle S_1,\cdots, S_k\rangle$
 the intersection of all the ratio-covarieties which contain the set $\{S_1,\cdots, S_k\}$ and they have the set $\Delta(F,m)$ as minimum.
 
 Note that $\sA(F,m)$ is a ratio-covariety with the above features. Next, by applying Lemma \ref{lemma51}, we obtain the following result.
 \begin{proposition}\label{proposition54}If  $S_1,\cdots, S_k$ are numerical semigroups with multiplicity $m$ and $F=\max \{\F(S_1),\cdots, \F(S_k)\},$ then  $\langle S_1,\cdots, S_k\rangle$ is the  least ratio-covariety  containing  the set $\{S_1,\cdots, S_k \}$ and having the set $\Delta(F,m)$ as minimum. 
 \end{proposition}

We will call to $\langle S_1,\cdots, S_k\rangle$ the {\it ratio-covariety generated }by $\{ S_1,\cdots, S_k\}.$ Our next aim will be to present an algorithm which allow us to compute all the elements of $\langle S_1,\cdots, S_k\rangle.$

For every $i\in \{1,\cdots,k\},$ define recursively
the following sequence: 
\begin{itemize}
	\item $S_i^0=S_i,$
	\item $S_i^{n+1}=\left\{\begin{array}{lcl}
		S_i^n \backslash \{\r(S_i^n)\} & &\mbox{if }\,\,  S_i^n\neq \Delta(F,m),\\
		\Delta(F,m) &  &\mbox{otherwise.}
	\end{array}
	\right.$
\end{itemize}

The following result has an immediate proof.
\begin{lemma}\label{lemma55}
For every $i\in \{1,\cdots,k\}$ there is $P_i=\min\{n\in \N\mid S_i^n=\Delta(F,m)\}.$
\end{lemma}
For all $i\in \{1,\cdots, k\},$ set $\Omega(S_i)=\{S_i^0,\cdots, S_i^{P_i}\}.$ Observe that $\Delta(F,m)=S_i^{P_i}\subsetneq S_i^{P_i-1}\subsetneq \cdots \subsetneq S_i^0=S_i $ and $\sharp(S_i^n \backslash S_i^{n+1})=1$ for each  $n\in \{0,\cdots, P_i-1\}.$ \\

In the next proposition we give the previously announced algorithmic procedure.

\begin{proposition}\label{proposition56}
Let $S_1,\cdots, S_k$ be numerical semigroups with multiplicity $m$ and  $F=\max \{\F(S_1),\cdots, \F(S_k) \}.$ Then $\langle S_1,\cdots, S_k\rangle=\displaystyle\{\bigcap_{b\in B}T_b\mid \emptyset \neq B \subseteq \{1,\cdots,k\} \mbox{ and }T_b\in \Omega(S_b) \mbox{ for all }b\in B\}.$
\end{proposition}
\begin{proof}
	To prove the proposition, it will be enough to see that $\sR=\displaystyle\{\bigcap_{b\in B}T_b\mid \emptyset \neq B \subseteq \{1,\cdots,k\} \mbox{ and }T_b\in \Omega(S_b) \mbox{ for all }b\in B\}$ is a ratio-covariety with minimum $\Delta(F,m).$ It is clear that $\min(\sR)=\Delta(F,m).$ Also,  it is easy to demonstrate that the intersection of two elements belonging to $\Omega(S_i)$ is also an element of $\Omega(S_i).$ Therefore, the intersection of two elements of $\sR$ is again an element of $\sR.$
	We prove next that if $T\in \sR$ and $T\neq \Delta(F,m),$ then $T\backslash \{\r(T)\}\in \sR.$ In fact, if $T\in \sR,$ then there is $\emptyset \neq  B \subseteq \{1,\cdots,k\}$ and there exists $T_b\in \Omega(S_b)$ for all $b\in B$ such that $T=\displaystyle \bigcap_{b\in B}T_b.$ As $\r(T)\in T,$ then $\r(T)\in S_b$ for all $b\in B.$ For every $b\in B,$ denote by $T'_b=\Delta(F,m)\cup \{x\in T_b\mid x> \r(T)\}.$ It is verified that $T'_b\in \Omega(S_b)$ for all $b\in B$ and $T\backslash \{\r(T)\}=\displaystyle \bigcap_{b\in B}T'_b.$ Therefore, $T\backslash \{\r(T)\}\in \sR.$
\end{proof}

From Proposition \ref{proposition56}, follows the following result.

\begin{corollary}\label{corollary57}
	If $S$ is a numerical semigroup, then $\langle S\rangle=\Rat$-$\Cad(S).$
\end{corollary}

We finish, illustring the content of Proposition \ref{proposition56} with an example.
\begin{example}\label{example58}
	Let $S_1=\langle 5,7,9\rangle=\{0,5,7,9,10,12,14,\rightarrow\}$ and $S_2=\langle 5,6,8\rangle=\{0,5,6,8,10,\rightarrow\}.$ Then $13=\max\{\F(S_1),\F(S_2)\}$ and so $\Omega(S_1)=\{S_1,S_1\backslash \{7\},\\
	 S_1\backslash \{7,9\}, S_1\backslash \{7,9,12\} \}$ and $\Omega(S_2)=\{S_2,S_2\backslash \{6\}, S_2\backslash \{6,8\}, S_2\backslash \{6,8,11\},\\
	  S_2\backslash \{6,8,11,12\}, S_2\backslash \{6,8,11,12,13\} \}.$ By applying Proposition \ref{proposition56}, we have that 
	$$
	\langle S_1,S_2\rangle=\Omega(S_1)\cup \Omega(S_2)\cup \{T_1\cap T_2\mid T_1 \in \Omega(S_1) \mbox{ and }T_2 \in \Omega(S_2)\}.
	$$
	
\end{example}

\end{document}